\documentclass[11pt, oneside]{amsart}   	
\usepackage{geometry}                		
\usepackage[parfill]{parskip}    		
\usepackage{graphicx}
\usepackage{amssymb,amsmath}
\newtheorem{theorem}{Theorem}
\newtheorem{lemma}[theorem]{Lemma}

\theoremstyle{definition}
\newtheorem{remark}[theorem]{Remark}
\newtheorem{example}[theorem]{Example}
\newtheorem{problem}[theorem]{Problem}
\newtheorem{definition}[theorem]{Definition}
\usepackage{epstopdf}
\numberwithin{equation}{section}
\numberwithin{theorem}{section}

\author[N.~Zorii]{Natalia Zorii}
\address{Institute of Mathematics of
National Academy of Sciences of Ukraine, Tereshchenkivska 3, 01601,
Kyiv-4, Ukraine}
\email{natalia.zorii@gmail.com}

\thanks{The author expresses her gratitude to Erwin Schr\"{o}dinger International Institute for providing  conducive research atmosphere during her stay when part of this manuscript was written.}

\begin{document}

\title[Gauss variational problem for condensers with touching plates]{Constrained Gauss variational problem\\ for condensers with touching plates}


\begin{abstract}
We study a constrained minimum energy problem with an external field relative to the $\alpha$-Riesz kernel $|x-y|^{\alpha-n}$ of an arbitrary order $\alpha\in(0,n)$ for a generalized condenser $\mathbf A=(A_1,A_2)$ with touching oppositely-charged plates in~$\mathbb R^n$, $n\geqslant2$. Conditions sufficient for the solvability of the problem are obtained. Our arguments are mainly based on the definition of an appropriate metric structure on a set of vector measures associated with~$\mathbf A$ and the establishment of a completeness theorem for the corresponding metric space.
\end{abstract}

\maketitle

\section{Introduction}This paper is devoted to the well-known Gauss variational problem of minimizing the $\alpha$-Riesz energy, $\alpha\in(0,n)$, in the presence of an external field, treated for a generalized condenser~$\mathbf A$ with touching oppositely-charged plates $A_1,A_2\subset\mathbb R^n$, $n\geqslant2$. In the case where the Euclidean distance ${\rm dist}(A_1,A_2)$ between~$A_1$ and~$A_2$ is nonzero (which might happen if $A_1$ and~$A_2$ touch each other {\it only\/} at the Alexandroff point~$\omega_{\mathbb R^n}$), a fairly complete investigation of this problem has been provided in~\cite{ZPot2,ZPot3} (see also the bibliography therein; see~Section~\ref{sec-Gauss} below for a short review).

However,  the results obtained in~\cite{ZPot2,ZPot3} and the approach developed are no longer valid if ${\rm dist}(A_1,A_2)=0$ (e.g, if $A_1$ and~$A_2$ touch each other at a {\it finite\/} point $x\in\mathbb R^n$). Then the infimum of the Gauss functional can not, in general, be attained among the admissible measures. Using the electrostatic interpretation, which is possible for the Coulomb kernel $|x-y|^{-1}$ on~$\mathbb R^3$, a short-circuit between~$A_1$ and~$A_2$ might occur. Therefore, it is meaningful to ask what kind of additional requirements on the charges (measures) under consideration would prevent this phenomenon.

A natural idea, to be exploited below, is to impose an upper constraint on vector measures associated with~$\mathbf A$ so that the infimum of the Gauss functional over the corresponding (narrower) class of constrained admissible vector measures would be already an actual minimum. See~Section~\ref{sec-Gauss-c} for a precise formulation of the constrained problem; as for the history of the question, cf.~Remarks~\ref{rem-3}--\ref{rem-5}.

A statement on the solvability of the constrained Gauss variational problem is given by Theorem~\ref{th-main}, the main result of the study. Its proof is based on the definition of an appropriate metric structure on a set of vector measures associated with~$\mathbf A$ and the establishment of a completeness theorem for the corresponding metric space (see~Theorem~\ref{complete}). The results obtained are illustrated by Example~\ref{ex-2}.

\section{Preliminaries}

Let $\mathrm X$ be a locally compact Hausdorff space, to be
specified below, and $\mathfrak M(\mathrm X)$ the linear
space of all real-valued scalar Radon measures~$\mu$ on~$\mathrm X$, equipped with the
{\it vague\/} topology, i.e.~the topology of
pointwise convergence on the class $\mathrm C_0(\mathrm X)$ of all real-valued
continuous functions on~$\mathrm X$ with compact
support. We denote by~$\mu^+$ and~$\mu^-$ the positive and the negative parts in the Hahn--Jordan decomposition of a measure $\mu\in\mathfrak M(\mathrm X)$, respectively, and by $S^\mu_{\mathrm X}$ its support. These and other notions of the theory of measures and integration in a locally compact space, to be used throughout the paper, can be found in~\cite{B2,E2} (see also~\cite{F1} for a short review).

A {\it
kernel\/} $\kappa(x,y)$ on $\mathrm X$ is a symmetric, lower
semicontinuous function $\kappa:\mathrm X\times\mathrm
X\to[0,\infty]$. Given $\mu,\mu_1\in\mathfrak M(\mathrm X)$, let
$E_\kappa(\mu,\mu_1)$ and $U_\kappa^\mu(\cdot)$ denote the {\it mutual
energy\/} and the {\it potential\/} relative to the kernel~$\kappa$,
respectively, i.e.
\begin{align*}
E_\kappa(\mu,\mu_1)&:=\int\kappa(x,y)\,d(\mu\otimes\mu_1)(x,y),\\
U_\kappa^\mu(x)&:=\int\kappa(x,y)\,d\mu(y),\quad x\in\mathrm X.
\end{align*}
(When introducing notation, we assume
the corresponding object on the right to be well defined --- as a finite number or $\pm\infty$.)

For $\mu=\mu_1$, the mutual energy $E_\kappa(\mu,\mu_1)$ defines the
{\it energy\/} $E_\kappa(\mu):=E_\kappa(\mu,\mu)$. Let $\mathcal E_\kappa(\mathrm X)$ consist
of all $\mu\in\mathfrak M(\mathrm X)$ whose energy $E_\kappa(\mu)$ is finite.

 Having denoted by $\mathfrak M^+(\mathrm X)$ the convex cone of all nonnegative $\mu\in\mathfrak M(\mathrm X)$, we write $\mathcal E^+_\kappa(\mathrm X):=\mathfrak M^+(\mathrm X)\cap\mathcal E_\kappa(\mathrm X)$. Given a set $B\subset\mathrm X$, $B\ne\mathrm X$, let $\mathfrak M^+(B;\mathrm X)$ consist of all $\mu\in\mathfrak M^+(\mathrm X)$ concentrated in~$B$, and let $\mathcal E^+_\kappa(B;\mathrm X):=\mathcal E_\kappa(\mathrm X)\cap\mathfrak M^+(B;\mathrm X)$.

 Observe that, if $B$ is closed, then $\mu\in\mathfrak M^+(\mathrm X)$ belongs to~$\mathfrak M^+(B;\mathrm X)$ if and only if the set $\mathrm X\setminus B$ is $\mu$-negligible (or, equivalently, if $S^\mu_{\mathrm X}\subset B$). Furthermore, then $\mathfrak M^+(B;\mathrm X)$ and $\mathcal E^+_\kappa(B;\mathrm X)$ are closed in the induced vague topology (see, e.g.,~\cite{F1}).

Let~$C_\kappa(B)$ be the {\it interior capacity\/} of~$B$
relative to the kernel~$\kappa$, given by
\begin{equation*}\label{cap-def}C_\kappa(B):=\bigl[\inf_{\mu\in\mathcal
E_\kappa^+(B;\mathrm X): \ \mu(B)=1}\,E_\kappa(\mu)\bigr]^{-1};\end{equation*}
see, e.g., \cite{F1,O}. Then $0\leqslant C_\kappa(B)\leqslant\infty$. (Here, as usual, the
infimum over the empty set is taken to be~$+\infty$. We also put
$1\bigl/(+\infty)=0$ and $1\bigl/0=+\infty$.)

A kernel~$\kappa$ is called {\it strictly positive definite\/} if the energy $E_\kappa(\mu)$, $\mu\in\mathfrak M(\mathrm X)$, is
nonnegative whenever defined and $E_\kappa(\mu)=0$ implies $\mu=0$.
Then $\mathcal E_\kappa(\mathrm X)$ forms a pre-Hil\-bert space with the
scalar product $E_\kappa(\mu,\mu_1)$ and the norm
$\|\mu\|_\kappa:=\sqrt{E_\kappa(\mu)}$ (see~\cite{F1}). The topology
on~$\mathcal E_\kappa(\mathrm X)$ defined by~$\|\cdot\|_\kappa$ is said to be {\it
strong\/}.

Following Fuglede~\cite{F1}, we call a strictly positive definite
kernel~$\kappa$ {\it perfect\/} if any strong Cauchy sequence in~$\mathcal E_\kappa^+(\mathrm X)$ converges strongly and, in addition, the strong topology on~$\mathcal E_\kappa^+(\mathrm X)$ is finer than the induced vague topology on~$\mathcal E_\kappa^+(\mathrm X)$. Note that then $\mathcal E_\kappa^+(\mathrm X)$ is a strongly complete metric space.

\section{Unconstrained and constrained Gauss variational problems}

Throughout the paper, let $n\geqslant2$, $n\in\mathbb N$, and $\alpha\in(0,n)$ be fixed. In $\mathrm X=\mathbb R^n$, consider the $\alpha$-{\it Riesz kernel\/} $\kappa_\alpha(x,y):=|x-y|^{\alpha-n}$ of order~$\alpha$, where $|x-y|$ denotes the Euclidean distance between~$x$ and~$y$ in~$\mathbb R^n$.
The $\alpha$-Riesz kernel is known to be strictly positive definite and, moreover, perfect
(see~\cite{D1,D2}); hence, the metric space $\mathcal E^+_{\kappa_\alpha}(\mathbb R^n)$ is complete in the induced strong topology. However, by
Cartan~\cite{Car} (see also \cite[Theorem~1.19]{L}), the whole pre-Hilbert space~$\mathcal
E_{\kappa_\alpha}(\mathbb R^n)$ for $\alpha\in(1,n)$ is strongly incomplete (compare with Theorem~\ref{complete} and Remark~\ref{rem-str} below).

From now on we shall write simply $\alpha$ instead
of~$\kappa_\alpha$ if it serves as an index. E.g.,
$C_\alpha(\cdot)=C_{\kappa_\alpha}(\cdot)$ denotes the
$\alpha$-Riesz interior capacity of a set. An expression $\mathcal U(x)$, involving a variable point $x\in\mathbb R^n$, is said to subsist {\it nearly everywhere\/} (n.e.) in a set $B\subset\mathbb R^n$ if $C_\alpha(N)=0$, where $N$ consists of all $x\in B$ for which $\mathcal U(x)$ fails to hold.

\subsection{Generalized condensers. Vector measures and their $\alpha$-Riesz energies}\label{sec-cond}

Given $B\subset\mathbb R^n$, write $B^c:=\mathbb R^n\setminus B$. Recall that a ({\it standard\/}) {\it condenser\/} in~$\mathbb R^n$ is usually meant as an ordered pair of nonempty, closed (though not necessarily compact), nonintersecting sets in~$\mathbb R^n$. We extend this notion as follows.

\begin{definition}\label{def-cond} An ordered pair $\mathbf A:=(A_1,A_2)$ of nonempty sets in $\mathbb R^n$ is called a {\it generalized condenser\/} if the following two conditions are fulfilled for every $i=1,2$:
\begin{itemize}
\item[\rm (a)] $A_i\subset D_i$, {\it where\/} $D_i:=\bigl({\rm C\ell}_{\mathbb R^n}A_j\bigr)^c$, $j\ne i$;
\item[\rm (b)] $A_i$ {\it is closed in the relative topology of the\/} ({\it open\/}) {\it set\/}~$D_i$.
\end{itemize}\end{definition}

Observe that the notion of a generalized condenser $\mathbf A=(A_1,A_2)$ is reduced to that of a standard one if and only if the sets~$A_i$, $i=1,2$, are closed in~$\mathbb R^n$.

In the example below, $n=3$ and $\overline{B}(x,1)$ is the closed three-dimensional ball of radius $1$ centered at $x\in\mathbb R^3$.

\begin{example}\label{ex-1} Consider $\overline{B}(\xi_1,1)$ and $\overline{B}(\xi_2,1)$ with $\xi_1=(0,0,0)$ and $\xi_2=(2,0,0)$; these balls intersect each other at $\xi_0=(1,0,0)$. Then the sets $A_i:=\overline{B}(\xi_i,1)\setminus\{\xi_0\}$, $i=1,2$, satisfy both assumptions~(a) and~(b) from Definition~\ref{def-cond} and, hence, form a generalized condenser~$\mathbf A$ in~$\mathbb R^3$, which certainly is not a standard one.\end{example}

In all that follows, fix a generalized condenser $\mathbf A=(A_1,A_2)$ such that $A_i\ne D_i$ for all $i=1,2$. To avoid triviality, suppose
$\prod_{i=1,2}\,C_\alpha(A_i)>0$.

Let $\mathfrak M^+(\mathbf A)$ stand for the Cartesian product $\prod_{i=1,2}\,\mathfrak M^+(A_i;D_i)$, where $D_i$ is thought of as a locally compact space. Then $\boldsymbol{\nu}\in\mathfrak M^+(\mathbf A)$ is a nonnegative {\it vector measure\/} $(\nu^i)_{i=1,2}$ with the components $\nu^i\in\mathfrak M^+(A_i;D_i)$; it is said to be {\it associated\/} with the condenser~$\mathbf A$.

\begin{definition}\label{def-vague}
  The $\mathbf A$-{\it vague\/} topology on $\mathfrak M^+(\mathbf A)$ is the topology of the product space $\prod_{i=1,2}\,\mathfrak M^+(A_i;D_i)$, where each of the factors $\mathfrak M^+(A_i;D_i)$, $i=1,2$, is endowed with the vague topology induced from~$\mathfrak M(D_i)$.\end{definition}

As $A_i$ is closed in $D_i$, $\mathfrak M^+(\mathbf A)$ is $\mathbf A$-vaguely closed. Besides, since every $\mathfrak M(D_i)$ is Hausdorff, so is $\mathfrak M^+(\mathbf A)$ (see~\cite[Chapter~3, Theorem~5]{K}). Hence, an $\mathbf A$-vague limit of any $\{{\boldsymbol\nu}_k\}_{k\in\mathbb N}\subset\mathfrak M^+(\mathbf A)$ belongs to~$\mathfrak M^+(\mathbf A)$ and is unique (provided it exists).

 If $\boldsymbol{\nu}\in\mathfrak M^+(\mathbf A)$ and a vector-valued function
$\boldsymbol{u}=(u_i)_{i=1,2}$ with the $\nu^i$-measurable components $u_i:A_i\to[-\infty,\infty]$ are given, then we write $\langle\boldsymbol{u},\boldsymbol{\nu}\rangle:=\sum_{i=1,2}\,\int u_i\,d\nu^i$.

We call $A_1$ and $A_2$ the {\it positive\/} and the {\it negative plates\/} of~$\mathbf A$, respectively. In accordance with the electrostatic interpretation of a condenser, assume that the interaction between the charges lying on the conductors~$A_i$, $i=1,2$, is characterized by the matrix $(s_is_j)_{i,j=1,2}$, where
\[ s_i:={\rm sign}\,A_i=\left\{
\begin{array}{rll} +1 & \mbox{if} & i=1,\\ -1 & \mbox{if} & i=2.\\ \end{array} \right. \]
Then the {\it $\alpha$-Riesz mutual energy\/} of $\boldsymbol{\nu},\boldsymbol{\nu}_1\in\mathfrak M^+(\mathbf A)$ is given formally by
\begin{equation}\label{env}E_\alpha(\boldsymbol{\nu},\boldsymbol{\nu}_1):=\sum_{i,j=1,2}\,s_is_j\int|x-y|^{\alpha-n}\,d(\nu^i\otimes\nu_1^j)(x,y).\end{equation}
For $\boldsymbol{\nu}=\boldsymbol{\nu}_1$, $E_\alpha(\boldsymbol{\nu},\boldsymbol{\nu}_1)$ defines  the {\it $\alpha$-Riesz energy\/} $E_\alpha(\boldsymbol{\nu}):=E_\alpha(\boldsymbol{\nu},\boldsymbol{\nu})$ of~$\boldsymbol{\nu}$. We denote by $\mathcal E_\alpha^+(\mathbf A)$ the set of all $\boldsymbol{\nu}\in\mathfrak M^+(\mathbf A)$ whose energy $E_\alpha(\boldsymbol{\nu})$ is finite.

\subsection{Metric structure on classes of vector measures}\label{sec-Metric}
Let $\breve{\mathfrak M}^+(\mathbf A)$ consist of all $\boldsymbol{\nu}\in\mathfrak M^+(\mathbf A)$ such that each of its components~$\nu^i$, $i=1,2$, can be extended to a Radon measure on~$\mathbb R^n$ (denote it again by~$\nu^i$) by setting
\[\nu^i(\varphi):=\langle\chi_{D_i}\varphi,\nu^i\rangle\quad\text{for all \ }\varphi\in\mathrm C_0(\mathbb R^n),\]
where $\chi_{D_i}$ is the characteristic function of~$D_i$. A sufficient condition for $\boldsymbol{\nu}\in\mathfrak M^+(\mathbf A)$ to belong to~$\breve{\mathfrak M}^+(\mathbf A)$ is that $\nu^i(A_i)<\infty$ for all $i=1,2$. Also note that
\begin{equation}\label{proper}\breve{\mathfrak M}^+(\mathbf A)=\mathfrak M^+(\mathbf A)\iff\text{$\mathbf A$ is standard};\end{equation}
otherwise,
$\breve{\mathfrak M}^+(\mathbf A)$ forms a proper subset of~$\mathfrak M^+(\mathbf A)$ that is not $\mathbf A$-vaguely closed.

For any $\boldsymbol{\nu}\in\breve{\mathfrak M}^+(\mathbf A)$, write
\begin{equation}\label{defR}R\boldsymbol{\nu}:=\sum_{i=1,2}\,s_i\nu^i;\end{equation}
then $R\boldsymbol{\nu}$ is a {\it signed\/} scalar Radon measure on~$\mathbb R^n$. Since $A_1\cap A_2=\varnothing$, $R$ is a one-to-one mapping between $\breve{\mathfrak M}^+(\mathbf A)$ and its $R$-image, \begin{equation*}R\bigl(\breve{\mathfrak M}^+(\mathbf A)\bigr)=\bigl\{\nu\in\mathfrak M(\mathbb R^n): \ \nu^+\in\mathfrak M^+(A_1;D_1), \ \nu^-\in\mathfrak M^+(A_2;D_2)\bigr\}.\end{equation*}

\begin{lemma}\label{l-Ren} For any\/
$\boldsymbol{\nu},\boldsymbol{\nu}_1\in\breve{\mathfrak M}^+(\mathbf A)$, $E_\alpha(\boldsymbol{\nu},\boldsymbol{\nu}_1)$ is well defined if and only if so is\/ $E_\alpha(R\boldsymbol{\nu},R\boldsymbol{\nu}_1)$, and then they coincide:
\begin{equation}\label{Re}E_\alpha(\boldsymbol{\nu},\boldsymbol{\nu}_1)=E_\alpha(R\boldsymbol{\nu},R\boldsymbol{\nu}_1).\end{equation}\end{lemma}

\begin{proof} Indeed, this can be obtained directly from~(\ref{env}) and (\ref{defR}).\end{proof}

In view of the strict positive definiteness of the $\alpha$-Riesz kernel, Lemma~\ref{l-Ren} yields that  $E_\alpha(\boldsymbol{\nu})$, $\boldsymbol{\nu}\in\breve{\mathfrak M}^+(\mathbf A)$, is ${}\geqslant0$ whenever defined, and it is zero only for $\boldsymbol{\nu}=\boldsymbol{0}$.
Write $\breve{\mathcal E}_\alpha^+(\mathbf A):=\mathcal E_\alpha^+(\mathbf A)\cap\breve{\mathfrak M}^+(\mathbf A)$. Having defined
\begin{equation*}\label{metr-def}\|\boldsymbol{\nu}-\boldsymbol{\nu}_1\|_{\breve{\mathcal E}^+_{\alpha}(\mathbf A)}:=\Bigl[\sum_{i,j=1,2}\,s_is_jE_\alpha(\nu^i-\nu_1^i,\nu^j-\nu_1^j)\Bigr]^{1/2}\quad\text{for all \ }\boldsymbol{\nu},\boldsymbol{\nu}_1\in\breve{\mathcal E}^+_\alpha(\mathbf A),\end{equation*}
we also see from~(\ref{Re}) by means of a straightforward calculation that, in fact,
\begin{equation}\label{isom}\|\boldsymbol{\nu}-\boldsymbol{\nu}_1\|_{\breve{\mathcal E}^+_\alpha(\mathbf A)}=\|R\boldsymbol{\nu}-R\boldsymbol{\nu}_1\|_\alpha,\end{equation}
so that $\breve{\mathcal E}^+_\alpha(\mathbf A)$ forms a metric space with the metric $\|\boldsymbol{\nu}-\boldsymbol{\nu}_1\|_{\breve{\mathcal E}^+_\alpha(\mathbf A)}$. Since, in consequence of~(\ref{isom}), $\breve{\mathcal E}^+_\alpha(\mathbf A)$ and its $R$-image are isometric, similar to the terminology in~$\mathcal E_\alpha(\mathbb R^n)$ we shall call the topology of the metric space $\breve{\mathcal E}^+_\alpha(\mathbf A)$ {\it strong}.

\subsection{Unconstrained $\mathbf{f}$-weighted minimum $\alpha$-Riesz energy problem}\label{sec-Gauss}
Given a locally compact space~$\mathrm X$, let $\mathrm\Phi(\mathrm X)$ consist of all lower semicontinuous
functions $\psi:\mathrm X\to(-\infty,\infty]$ such that
$\psi\geqslant0$ unless $\mathrm X$ is compact. Then for any $\psi\in\mathrm\Phi(\mathrm X)$, the map
\[\mu\mapsto\langle\psi,\mu\rangle,\quad\mu\in\mathfrak M^+(\mathrm X),\] is vaguely lower semicontinuous (see, e.g.,~\cite[Section~1.1]{F1}).

Fix a vector-valued function $\mathbf{f}=(f_i)_{i=1,2}$, where each $f_i:A_i\to[-\infty,\infty]$ is universally measurable and it is treated as an {\it external field\/} acting on the charges from $\mathfrak M^+(A_i;D_i)$. Then
the $\mathbf f$-{\it weighted $\alpha$-Riesz energy\/} of $\boldsymbol{\nu}\in\mathcal E_\alpha^+(\mathbf A)$ is defined by
\begin{equation}
\label{wen}G_{\alpha,\mathbf{f}}(\boldsymbol{\nu}):=E_\alpha(\boldsymbol{\nu})+2\langle\mathbf{f},\boldsymbol{\nu}\rangle;\end{equation}
$G_{\alpha,\mathbf{f}}(\mathbf{\cdot})$ is also known as the {\it Gauss functional\/} (see, e.g.,~\cite{O}). Let $\mathcal E_{\alpha,\mathbf{f}}^+(\mathbf A)$ consist of all $\boldsymbol{\nu}\in\mathcal E_\alpha^+(\mathbf A)$ with finite $G_{\alpha,\mathbf{f}}(\boldsymbol{\nu})$.

In this paper, we tacitly assume that one of the following Cases~I or~II holds:
\begin{itemize}
\item[\rm I.] {\it For every\/ $i=1,2$, $f_i\in\mathrm\Phi(A_i)$, where\/ $A_i$ is thought of as a locally compact space\/};
\item[\rm II.] {\it For every\/ $i=1,2$, $f_i=s_iU_\alpha^\zeta\bigl|_{A_i}$, where a\/ {\rm(}signed\/{\rm)} scalar measure\/ $\zeta\in\mathcal E_\alpha(\mathbb R^n)$ is given\/}.
\end{itemize}
For any $\boldsymbol{\nu}\in\breve{\mathcal E}^+_{\alpha}(\mathbf A)$, $G_{\alpha,\mathbf f}(\boldsymbol{\nu})$ is then well defined in both Cases~I and~II.
Furthermore, if Case~II takes place, then, by~(\ref{wen}) and~(\ref{Re}),
\begin{align}\label{C2}G_{\alpha,\mathbf f}(\boldsymbol{\nu})&=\|R\boldsymbol{\nu}\|^2_\alpha+2\sum_{i=1,2}\,s_iE_\alpha(\zeta,\nu^i)\\{}&=\|R\boldsymbol{\nu}\|^2_\alpha+2E_\alpha(\zeta,R\boldsymbol{\nu})=
\|R\boldsymbol{\nu}+\zeta\|^2_\alpha-\|\zeta\|_\alpha^2\notag\end{align}
and, consequently,
\begin{equation}\label{GII}-\infty<-\|\zeta\|_\alpha^2\leqslant G_{\alpha,\mathbf f}(\boldsymbol{\nu})<\infty\quad\text{for all \ }\boldsymbol{\nu}\in\breve{\mathcal E}^+_{\alpha}(\mathbf A).\end{equation}

Also fix a numerical vector $\mathbf a=(a_i)_{i=1,2}$ with $a_i>0$ and a vector-valued function $\mathbf{g}=(g_i)_{i=1,2}$, where all the $g_i:D_i\to(0,\infty)$ are continuous and such that
\begin{equation}\label{infg}g_{i,\inf}:=\inf_{x\in
A_i}\,g_i(x)>0.\end{equation} Write
\begin{align*}\mathfrak M^+(\mathbf A,\mathbf a,\mathbf g)&:=\bigl\{\boldsymbol{\nu}\in\mathfrak M^+(\mathbf A): \ \langle g_i,\nu^i\rangle=a_i\quad\text{for all \ }i=1,\,2\bigr\},\\
\mathcal E^+_{\alpha,\mathbf f}(\mathbf A,\mathbf a,\mathbf g)&:=\mathfrak M^+(\mathbf A,\mathbf a,\mathbf g)\cap\mathcal E^+_{\alpha,\mathbf f}(\mathbf A),\\
G_{\alpha,\mathbf f}(\mathbf A,\mathbf a,\mathbf g)&:=\inf_{\boldsymbol{\nu}\in\mathcal E^+_{\alpha,\mathbf f}(\mathbf A,\mathbf a,\mathbf g)}\,G_{\alpha,\mathbf f}(\boldsymbol{\nu}).\end{align*}

Observe that, because of~(\ref{infg}),
\[\nu^i(A_i)\leqslant a_ig_{i,\inf}^{-1}<\infty\quad\text{for all \ }\boldsymbol{\nu}\in\mathfrak M^+(\mathbf A,\mathbf a,\mathbf g)\]
and, therefore,
\begin{equation}\label{inlusions}\mathfrak M^+(\mathbf A,\mathbf a,\mathbf g)\subset\breve{\mathfrak M}^+(\mathbf A),\quad\mathcal E^+_{\alpha,\mathbf f}(\mathbf A,\mathbf a,\mathbf g)\subset
\breve{\mathcal E}^+_\alpha(\mathbf A).\end{equation}
Combined these with Lemma~\ref{l-Ren} and the fact that a lower semicontinuous function is bounded from below on a compact set, in Case~I we obtain
\[G_{\alpha,\mathbf f}(\mathbf A,\mathbf a,\mathbf g)>-\infty.\]
The same holds true in Case~II as well, which is obvious from~(\ref{GII}) and~(\ref{inlusions}).

If the class $\mathcal E^+_{\alpha,\mathbf f}(\mathbf A,\mathbf a,\mathbf g)$ is nonempty or, equivalently, if
\begin{equation}\label{Gf}G_{\alpha,\mathbf f}(\mathbf A,\mathbf a,\mathbf g)<\infty,\end{equation}
then the following (unconstrained) $\mathbf{f}$-weighted minimum $\alpha$-Riesz energy problem, also known as the {\it Gauss variational problem\/} (see~\cite{Gauss,O}), makes sense.

\begin{problem}\label{pr1} Does there exist $\boldsymbol{\lambda}_{\mathbf A}\in\mathcal E^+_{\alpha,\mathbf f}(\mathbf A,\mathbf a,\mathbf g)$ with $G_{\alpha,\mathbf f}(\boldsymbol{\lambda}_{\mathbf A})=G_{\alpha,\mathbf f}(\mathbf A,\mathbf a,\mathbf g)$?\end{problem}

\begin{remark} Analysis similar to that for a standard condenser (cf.~Lemma~6.2 in~\cite{ZPot2}) shows that assumption~(\ref{Gf}) is equivalent to the following one:
\[f_i(x)<\infty\quad\text{n.e.~in \ }A_i,\quad i=1,2.\]
In turn, this yields that (\ref{Gf}) holds automatically whenever Case~II takes place, for the $\alpha$-Riesz potential of $\zeta\in\mathcal E_\alpha(\mathbb R^n)$ is finite n.e.~in~$\mathbb R^n$.\end{remark}

\begin{remark}\label{r-2}In the case where every $A_i$ is compact in~$D_i$ (i.e., $\mathbf A$ is a compact standard condenser) and Case~I takes place, the solvability  of Problem~\ref{pr1} can easily be established by exploiting the $\mathbf A$-va\-gue topology only, since then $\mathfrak M^+(\mathbf A,\mathbf a,\mathbf g)$ is $\mathbf A$-va\-guely compact, while $G_{\alpha,\mathbf f}(\mathbf{\cdot})$ is $\mathbf A$-va\-guely lower semicontinuous on~$\mathcal E^+_{\alpha,\mathbf f}(\mathbf A)$ (see~\cite[Theorem~2.30]{O}). However, these arguments break down if any of the two requirements is not satisfied, and then Problem~\ref{pr1} becomes rather nontrivial. E.g., $\mathfrak M^+(\mathbf A,\mathbf a,\mathbf g)$ is no longer $\mathbf A$-vaguely compact if some of the~$A_i$ is noncompact in~$D_i$.\end{remark}

\begin{remark}Assume that $\mathbf A$ is still a standard condenser, though now, in contrast to Remark~\ref{r-2}, its plates might be noncompact in~$\mathbb R^n$. Under the assumption
\begin{equation}\label{distpos}{\rm dist}\,(A_1,A_2):=\inf_{x\in A_1, \ y\in A_2}\,|x-y|>0,\end{equation}
in~\cite{ZPot2,ZPot3} we worked out an approach based on both the $\mathbf A$-vague and the strong topologies on~$\mathcal E_\alpha^+(\mathbf A)$ and a certain strong completeness result, which made it possible to provide a fairly complete analysis of Problem~\ref{pr1}. In more detail, it has been shown that, if $g_i|_{A_i}$, $i=1,2$, are bounded from above, then, in both Cases~I and~II,  \begin{equation}\label{r-suff}C_\alpha(A_1\cup A_2)<\infty\end{equation} is sufficient for Problem~\ref{pr1} to be (uniquely) solvable for every~$\mathbf a$ (see~\cite[Theorem~8.1]{ZPot2}). However, if (\ref{r-suff}) does not hold, then, in general, there exists a vector~$\mathbf a'$ such that the Gauss variational problem admits no solution~\cite{ZPot2}. Therefore, it was interesting to give a description of the set of all vectors~$\mathbf a$ for which the
problem would be, nevertheless, solvable. Such a characterization has been established in~\cite{ZPot3}.\end{remark}

In the rest of the paper, except for Remark~\ref{rem-3}, we do not assume (\ref{distpos}) necessarily to hold.
Then the results obtained in~\cite{ZPot2,ZPot3} and the approach developed are no longer valid.
In particular, assumption~(\ref{r-suff}) does not guarantee anymore that $G_{\alpha,\mathbf f}(\mathbf A,\mathbf a,\mathbf g)$ is attained among $\boldsymbol{\nu}\in\mathcal E^+_{\alpha,\mathbf f}(\mathbf A,\mathbf a,\mathbf g)$. Using the electrostatic interpretation, a short-circuit between the touching oppositely-charged plates of the condenser might occur. Therefore, it is meaningful to ask what kind of additional requirements on the measures under consideration would prevent this phenomenon, and a solution to the corresponding $\mathbf{f}$-weighted minimum $\alpha$-Riesz energy problem would, nevertheless, exist.

The idea discussed below is to find out such an upper constraint on the measures from $\mathfrak M^+(\mathbf A,\mathbf a,\mathbf g)$ which would not allow the "blow-up" effect between~$A_1$ and~$A_2$.

\subsection{Constrained $\mathbf{f}$-weighted minimum $\alpha$-Riesz energy problem}\label{sec-Gauss-c}

Let $\mathfrak C(\mathbf A)$ consist of all $\boldsymbol{\sigma}=(\sigma^i)_{i=1,2}\in\mathfrak M^+(\mathbf A)$ such that
\begin{equation}\label{g-mass}S^{\sigma^i}_{D_i}=A_i\text{ \ and \ }\langle g_i,\sigma^i\rangle>a_i\quad\text{for all \ }i=1,2;\end{equation}
these $\boldsymbol{\sigma}$ will serve as {\it constraints\/} for $\boldsymbol{\nu}\in\mathfrak M^+(\mathbf A)$. Given $\boldsymbol{\sigma}\in\mathfrak C(\mathbf A)$, write
\[\mathfrak M^{\boldsymbol{\sigma}}(\mathbf A):=\bigl\{\boldsymbol{\nu}\in\mathfrak M^+(\mathbf A): \ \nu^i\leqslant\sigma^i\quad\text{for all \ }i=1,\,2\bigr\},\]
where $\nu^i\leqslant\sigma^i$ means that $\sigma^i-\nu^i$ is a nonnegative scalar measure, and
\begin{align*}
\mathfrak M^{\boldsymbol{\sigma}}(\mathbf A,\mathbf a,\mathbf g)&:=\mathfrak M^{\boldsymbol{\sigma}}(\mathbf A)\cap\mathfrak M^+(\mathbf A,\mathbf a,\mathbf g),\\
\mathcal E^{\boldsymbol{\sigma}}_{\alpha,\mathbf f}(\mathbf A,\mathbf a,\mathbf g)&:=\mathfrak M^{\boldsymbol{\sigma}}(\mathbf A,\mathbf a,\mathbf g)\cap\mathcal E^+_{\alpha,\mathbf f}(\mathbf A).\end{align*}
Since $\mathcal E^{\boldsymbol{\sigma}}_{\alpha,\mathbf f}(\mathbf A,\mathbf a,\mathbf g)\subset\mathcal E^+_{\alpha,\mathbf f}(\mathbf A,\mathbf a,\mathbf g)$,
we get
\[-\infty<G_{\alpha,\mathbf f}(\mathbf A,\mathbf a,\mathbf g)\leqslant G^{\boldsymbol{\sigma}}_{\alpha,\mathbf f}(\mathbf A,\mathbf a,\mathbf g):=\inf_{\boldsymbol{\nu}\in\mathcal E^{\boldsymbol{\sigma}}_{\alpha,\mathbf f}(\mathbf A,\mathbf a,\mathbf g)}\,G_{\alpha,\mathbf f}(\boldsymbol{\nu})\leqslant\infty.\]

If the class $\mathcal E^{\boldsymbol{\sigma}}_{\alpha,\mathbf f}(\mathbf A,\mathbf a,\mathbf g)$ is nonempty or, equivalently, if
\begin{equation}\label{gauss-finite}G_{\alpha,\mathbf f}^{\boldsymbol{\sigma}}(\mathbf A,\mathbf a,\mathbf g)<\infty,\end{equation}
then the following constrained $\mathbf{f}$-weighted minimum $\alpha$-Riesz energy problem, also known as the {\it constrained Gauss variational problem\/}, makes sense.

\begin{problem}\label{pr2}Given $\boldsymbol{\sigma}\in\mathfrak C(\mathbf A)$, does there exist $\boldsymbol{\lambda}_{\mathbf A}^{\boldsymbol{\sigma}}\in\mathcal E^{\boldsymbol{\sigma}}_{\alpha,\mathbf f}(\mathbf A,\mathbf a,\mathbf g)$ with \[G_{\alpha,\mathbf f}(\boldsymbol{\lambda}^{\boldsymbol{\sigma}}_{\mathbf A})=G_{\alpha,\mathbf f}^{\boldsymbol{\sigma}}(\mathbf A,\mathbf a,\mathbf g)?\]\end{problem}

\begin{remark}\label{rem-3}Assume for a moment that (\ref{distpos}) holds. It has been shown by~\cite[Theorem~6.2]{Z9} that if, in addition,
$g_i|_{A_i}$, $i=1,2$, are bounded from above and conditions~(\ref{r-suff}) and~(\ref{gauss-finite}) are satisfied, then, in both Cases~I and~II, Problem~\ref{pr2} is (uniquely) solvable. But this does not remain true if requirement~(\ref{distpos}) is dropped.\end{remark}

\begin{remark}\label{rem-4}If $0<\alpha\leqslant2<n$, $a_1=a_2$, $\mathbf g=\mathbf 1$, $A_2$ is not $\alpha$-thin at~$\omega_{\mathbb R^n}$, $f_2=0$ and $\sigma^2=\infty$ (i.e., no external field and no constraint act on the measures concentrated in~$A_2$), then sufficient and/or necessary conditions for the solvability of Problem~\ref{pr2} have been established in~\cite{DHSZ}. Crucial to the arguments exploited in~\cite{DHSZ} is that, in this special case, Problem~\ref{pr2} can be reduced to the problem of minimizing the $f_1$-weighted $g^\alpha_{D_1}$-Green energy over the class $\mathcal E^+_{g^\alpha_{D_1}}(A_1;D_1)$. However, under the assumptions of the present study, such an observation is no longer valid.\end{remark}

\begin{remark}\label{rem-5}If $a_1=a_2$, $\mathbf g=\mathbf 1$, $\mathbf f=\mathbf 0$ and $A_i$, $i=1,2$, are bounded, then the constrained minimum logarithmic energy problem for a condenser with touching plates in~$\mathbb C$ has been investigated by Beckermann and Gryson (see~\cite[Theorem~2.2]{BC}). Our paper is related to the $\alpha$-Riesz
kernels, $0<\alpha<n$, in~$\mathbb R^n$, $n\geqslant2$, and the results obtained and the approaches developed are rather different from those in~\cite{BC}.\end{remark}

\section{Sufficient conditions for the solvability of Problem~\ref{pr2}}\label{sec-main}
Denote by $\overline{B}$ the closure of $B\subset\mathbb R^n$ in $\overline{\mathbb R^n}:=\mathbb R^n\cup\{\omega_{\mathbb R^n}\}$, the one-point compactification of~$\mathbb R^n$.

\begin{theorem}\label{th-main}
Let\/ $\mathbf A$, $\mathbf f$, $\mathbf g$ and\/ $\boldsymbol{\sigma}\in\mathfrak C(\mathbf A)$ possess the following four properties:
\begin{itemize}
\item[\rm{(a$'$)}] $\overline{A_1}\cap\overline{A_2}$ consists of at most one point, i.e., $\overline{A_1}\cap\overline{A_2}=\varnothing\vee\{x_0\}$ where\/ $x_0\in\overline{\mathbb R^n}$;
\item[\rm{(b$'$)}]$f_i(x)<\infty$ n.e.~in\/~$A_i$, $i=1,2$;
\item[\rm{(c$'$)}]$E_\alpha\bigl(\sigma^i\bigl|_{K_i}\bigr)<\infty$ for every compact\/ $K_i\subset A_i$, $i=1,2$;
\item[\rm{(d$'$)}]$\langle g_i,\sigma^i\rangle<\infty$, $i=1,2$.
\end{itemize}
Then, in both Cases\/~{\rm I} and\/~{\rm II}, Problem\/~{\rm\ref{pr2}} is uniquely solvable for every vector\/~$\mathbf a$.
\end{theorem}

The proof of Theorem~\ref{th-main} is given in Section~\ref{sec-proof-main}; it is based on Theorem~\ref{complete}, which provides a strong completeness result for metric subspaces of~$\breve{\mathcal E}^+_\alpha(\mathbf A)$.

\begin{example}\label{ex-2} Let $\mathbf A=(A_1,A_2)$ be as in Example~\ref{ex-1}. Having fixed $\alpha\in(0,3)$, assume that $\mathbf g=\mathbf 1$ and either Case~II holds or $f_i(x)<\infty$ n.e.~in~$A_i$, $i=1,2$. For any $\mathbf a=(a_i)_{i=1,2}$ define $\sigma^i:=c_im_3|_{A_i}$, where $c_i\in(a_i,\infty)$ is chosen  arbitrarily and $m_3$ denotes the $3$-dimensional Lebesgue measure on~$\mathbb R^3$. Then, by Theorem~\ref{th-main},  Problem~\ref{pr2} admits a solution; hence, no short-cir\-cuit between~$A_1$ and~$A_2$ occurs, though these conductors touch each other at the point~$\xi_0$ (see Example~\ref{ex-1}).\end{example}

\section{Strong completeness theorem for metric subspaces of $\breve{\mathcal E}^+_\alpha(\mathbf A)$}\label{sec-str}
Let $\mathfrak M^+(\mathbf A,\leqslant\!\mathbf a,\mathbf g)$ consist of all $\boldsymbol{\nu}\in\mathfrak M^+(\mathbf A)$ such that $\langle g_i,\nu^i\rangle\leqslant a_i$ for all $i=1,2$. In view of~(\ref{infg}),
\begin{equation}\label{bbbound}\nu^i(A_i)\leqslant a_ig_{i,\inf}^{-1}<\infty\quad\text{for all \ }\boldsymbol{\nu}\in\mathfrak M^+(\mathbf A,\leqslant\!\mathbf a,\mathbf g).\end{equation}
Hence, $\mathcal E^+_{\alpha}(\mathbf A,\leqslant\!\mathbf a,\mathbf g):=\mathcal E^+_{\alpha}(\mathbf A)\cap\mathfrak M^+(\mathbf A,\leqslant\!\mathbf a,\mathbf g)$
can be thought of as a metric subspace of $\breve{\mathcal E}^+_{\alpha}(\mathbf A)$; its topology will likewise be called {\it strong}.

\begin{theorem}\label{complete} Suppose that a generalized condenser\/ $\mathbf A$ satisfies condition\/ {\rm(a$'$)} of Theorem\/~{\rm\ref{th-main}}. Then the metric space\/ $\mathcal E^+_\alpha(\mathbf A,\leqslant\!\mathbf a,\mathbf g)$ is strongly complete and the strong topology on this space is finer than the induced\/ $\mathbf A$-vague topology.\end{theorem}

\begin{remark}\label{rem-str}In view of the fact that the metric space $\mathcal
E^+_\alpha(\mathbf A,\leqslant\!\mathbf a,\mathbf g)$ is isometric to its
$R$-image, Theorem~\ref{complete} has singled out a strongly
complete topological subspace of the pre-Hilbert space~$\mathcal
E_\alpha(\mathbb R^n)$, whose elements are signed Radon measures. This is of independent interest since, according to a well-known counterexample by Cartan, the whole pre-Hilbert space $\mathcal E_\alpha(\mathbb R^n)$ is, in general, strongly incomplete.\end{remark}

\subsection{Auxiliary results}
Based on the definition of the $\mathbf A$-vague topology (see Definition~\ref{def-vague}), we call a set $\mathfrak F\subset\mathfrak M^+(\mathbf A)$ {\it
$\mathbf A$-vaguely bounded\/} if, for every $i=1,2$ and every $\varphi\in\mathrm C_0(D_i)$,
\[\sup_{\boldsymbol{\nu}\in\mathfrak F}\,|\nu^i(\varphi)|<\infty.\]

\begin{lemma}\label{lem:vaguecomp} If\/ $\mathfrak F\subset\mathfrak
M^+(\mathbf A)$ is\/ $\mathbf A$-vaguely bounded, then it is\/ $\mathbf A$-vaguely relatively
compact.\end{lemma}

\begin{proof} Since by~\cite[Chapter~III, Section~2, Proposition~9]{B2} any
vaguely bounded part of~$\mathfrak M^+(D_i)$ is vaguely relatively compact,
the lemma follows from Tychonoff's theorem on the product of compact
spaces (see, e.g.,~\cite[Chapter~5, Theorem~13]{K}).\end{proof}

\begin{lemma}\label{lemma-rel-comp} $\mathfrak M^+(\mathbf A,\leqslant\!\mathbf a,\mathbf g)$ is\/ $\mathbf A$-vaguely bounded and\/ $\mathbf A$-vaguely closed; hence, it is\/ $\mathbf A$-vaguely compact.\end{lemma}

\begin{proof} Indeed, it is obvious from~(\ref{bbbound}) that $\mathfrak M^+(\mathbf A,\leqslant\!\mathbf a,\mathbf g)$ is $\mathbf A$-vaguely bounded. Fix  an arbitrary $\{\boldsymbol{\nu}_k\}_{k\in\mathbb N}\subset\mathfrak M^+(\mathbf A,\leqslant\!\mathbf a,\mathbf g)$; then, by Lemma~\ref{lem:vaguecomp}, it has an $\mathbf A$-va\-gue cluster point~$\boldsymbol{\nu}_0$. In fact, $\boldsymbol{\nu}_0\in\mathfrak M^+(\mathbf A)$, for $\mathfrak M^+(\mathbf A)$ is $\mathbf A$-vaguely closed. Choose a subsequence $\{\boldsymbol{\nu}_{k_m}\}_{m\in\mathbb N}$ of $\{\boldsymbol{\nu}_k\}_{k\in\mathbb N}$ that converges $\mathbf A$-vaguely to~$\boldsymbol{\nu}_0$. As $g_i$ is positive and continuous, we get
\[\langle g_i,\nu_0^i\rangle\leqslant\liminf_{m\to\infty}\,\langle g_i,\nu_{k_m}^i\rangle\leqslant a_i\quad\text{for all \ }i=1,2,\]
and the lemma follows.\end{proof}

\begin{lemma}\label{lemma-str-com}Assume that\/ $\mathbf A$ is a standard condenser; i.e., $\overline{A_1}\cap\overline{A_2}=\varnothing\vee\{\omega_{\mathbb R^n}\}$. Then the metric space\/ $\mathcal E^+_\alpha(\mathbf A)$ $\bigl({}=\breve{\mathcal E}^+_{\alpha}(\mathbf A)\bigr)$ is strongly complete. In more detail, any strong Cauchy sequence\/ $\{\boldsymbol{\nu}_k\}_{k\in\mathbb N}\subset\mathcal E^+_\alpha(\mathbf A)$ converges both strongly and\/ $\mathbf A$-vaguely to some\/ $\boldsymbol{\nu}_0\in\mathcal E^+_\alpha(\mathbf A)$, and this limit is unique.\end{lemma}

\begin{proof} It is clear from (\ref{proper}) that, for a standard~$\mathbf A$,
\[\mathcal E^+_\alpha(\mathbf A)=\breve{\mathcal E}^+_{\alpha}(\mathbf A).\]
Since $\breve{\mathcal E}^+_\alpha(\mathbf A)$ and $R\bigl(\breve{\mathcal E}^+_\alpha(\mathbf A)\bigr)$, the latter being treated as a metric subspace of the pre-Hilbert space $\mathcal E_\alpha(\mathbb R^n)$, are isometric to each other by~(\ref{isom}), the lemma follows from~\cite{ZUmzh} (see~Theorem~1 and Corollary~1 therein).\end{proof}

\subsection{Proof of Theorem~\ref{complete}} Fix a strong Cauchy sequence $\{\boldsymbol{\nu}_k\}_{k\in\mathbb N}\subset\mathcal E^+_\alpha(\mathbf A,\leqslant\!\mathbf a,\mathbf g)$. According to Lemma~\ref{lemma-rel-comp},
it has an $\mathbf A$-vague cluster point $\boldsymbol{\nu}_0\in\mathfrak M^+(\mathbf A,\leqslant\!\mathbf a,\mathbf g)$.
Let $\{\boldsymbol{\nu}_{k_m}\}_{m\in\mathbb N}$ be a (strong Cauchy) subsequence of $\{\boldsymbol{\nu}_k\}_{k\in\mathbb N}$ that converges $\mathbf A$-vaguely to~$\boldsymbol{\nu}_0$, i.e.
\begin{equation}\label{conv-vag}\nu_{k_m}^i\to\nu_0^i\quad\text{vaguely in \ }\mathfrak M(D_i), \ i=1,2.\end{equation}

We proceed by showing that $E_\alpha(\boldsymbol{\nu}_0)$ is finite, so that
\begin{equation}\label{belongs}\boldsymbol{\nu}_0\in\mathcal E^+_\alpha(\mathbf A,\leqslant\!\mathbf a,\mathbf g)\quad\bigl({}\subset\breve{\mathcal E}^+_{\alpha}(\mathbf A)\bigr),\end{equation}
and, moreover, $\boldsymbol{\nu}_{k_m}\to\boldsymbol{\nu}_0$ strongly as $m\to\infty$, i.e.
\begin{equation}\label{conv-str}\lim_{m\to\infty}\,\|\boldsymbol{\nu}_{k_m}-\boldsymbol{\nu}_0\|_{\breve{\mathcal E}^+_{\alpha}(\mathbf A)}=0.\end{equation}
To establish these assertions, it is enough to analyze the case
\begin{equation}\label{ex}\overline{A_1}\cap\overline{A_2}=\{x_0\}\quad\text{where \ }x_0\in\mathbb R^n,\end{equation} since otherwise they are obtained directly from Lemma~\ref{lemma-str-com}.

Consider the inversion~$I$ with respect to the $(n-1)$-dim\-ensional unit sphere centered at~$x_0$; namely, each point $x\ne
x_0$ is mapped to the point~$x^*$ on the ray through~$x$ which
issues from~$x_0$, determined uniquely by
\[|x-x_0|\cdot|x^*-x_0|=1.\]
This is a one-to-one, bicontinuous mapping of $\mathbb
R^n\setminus\{x_0\}$ onto itself; furthermore,
\begin{equation}\label{inv}|x^*-y^*|=\frac{|x-y|}{|x_0-x||x_0-y|}.\end{equation}
Extend it to a one-to-one,
bicontinuous map of $\overline{\mathbb R^n}$ onto itself by setting $I(x_0)=\omega_{\mathbb R^n}$.

To each signed scalar measure $\nu\in\mathfrak M(\mathbb R^n)$ with
$\nu\bigl(\{x_0\}\bigr)=0$ there corresponds the Kelvin transform
$\nu^*\in\mathfrak M(\mathbb R^n)$ by means of the
formula
\[d\nu^*(x^*)=|x-x_0|^{\alpha-n}\,d\nu(x),\quad x^*\in\mathbb R^n\]
(see~\cite{R} or \cite[Chapter IV, Section 5, n$^\circ$\,19]{L}). Then, in view of~(\ref{inv}),
\begin{equation*}\label{KP}U_\alpha^{\nu^*}(x^*)=|x-x_0|^{n-\alpha}U_\alpha^{\nu}(x),\quad x^*\in\mathbb R^n,\end{equation*}
and therefore
\begin{equation}\label{K}E_\alpha(\nu^*)=E_\alpha(\nu).\end{equation}
It is clear that the Kelvin transformation is additive and it is an involution, i.e.
\begin{align}\label{A}\bigl(\nu_1+\nu_2\bigr)^*&=
\nu_1^*+\nu_2^*,\\
\label{AA}(\nu^*)^*&=\nu.\end{align}

Write $A_i^*:=I\bigl(\overline{A_i}\bigr)\cap\mathbb R^n$, $i=1,2$; then $\mathbf A^*=(A_1^*,A_2^*)$ forms a standard condenser in~$\mathbb R^n$, which is obvious from~(\ref{ex}) and the above-mentioned properties of~$I$.

Applying the Kelvin transformation to each of the components of any given $\boldsymbol{\nu}=(\nu^i)_{i=1,2}\in\breve{\mathfrak M}^+(\mathbf A)$, we get $\boldsymbol{\nu}^*:=\bigl((\nu^i)^*\bigr)_{i=1,2}\in\mathfrak M^+(\mathbf A^*)$; and the other way around.
Based on Lemma~\ref{l-Ren} and relations~(\ref{isom}) and (\ref{K})--(\ref{AA}), we also see that
the $\alpha$-Riesz energy of $\boldsymbol{\nu}\in\breve{\mathfrak M}^+(\mathbf A)$ is well defined if and only if so is that of~$\boldsymbol{\nu}^*$, and then they coincide; and, furthermore,
\begin{equation}\label{pres}\|\boldsymbol{\nu}_1^*-\boldsymbol{\nu}_2^*\|_{\mathcal E^+_\alpha(\mathbf A^*)}=\|\boldsymbol{\nu}_1-\boldsymbol{\nu}_2\|_{\breve{\mathcal E}^+_\alpha(\mathbf A)}\quad\text{for all \ }\boldsymbol{\nu}_1,\boldsymbol{\nu}_2\in\breve{\mathcal E}^+_\alpha(\mathbf A).\end{equation}
Summarizing what has thus been observed, we conclude that the Kelvin transformation is a one-to-one, isometric mapping of $\breve{\mathcal E}^+_\alpha(\mathbf A)$ onto~$\mathcal E^+_\alpha(\mathbf A^*)$.

Let $\boldsymbol{\nu}_{k_m}$, $m\in\mathbb N$, and~$\boldsymbol{\nu}_0$ be as above. In view of~(\ref{bbbound}) and~(\ref{conv-vag}), for each $i=1,2$ one can apply \cite[Lemma~4.3]{L} to
$\nu_{k_m}^i$, $k\in\mathbb N$, and~$\nu^i_0$, and consequently
\begin{equation}\label{vague'}\boldsymbol{\nu}_{k_m}^*\to\boldsymbol{\nu}_0^*\quad\text{$\mathbf A$-vaguely as \
}m\to\infty.\end{equation}
But $\bigl\{\boldsymbol{\nu}_{k_m}^*\bigr\}_{m\in\mathbb N}$ is a strong Cauchy sequence in $\mathcal E_\alpha^+(\mathbf A^*)$, which is clear from~(\ref{pres}). This together with~(\ref{vague'}) implies, by Lemma~\ref{lemma-str-com}, that $\boldsymbol{\nu}_0^*\in\mathcal E_\alpha^+(\mathbf A^*)$ and
\[\lim_{m\to\infty}\,\|\boldsymbol{\nu}_{k_m}^*-\boldsymbol{\nu}_0^*\|_{\mathcal E_\alpha^+(\mathbf A^*)}=0.\]
Repeated application of~(\ref{pres}) then leads to relations~(\ref{belongs}) and~(\ref{conv-str}) as claimed.

In turn, (\ref{conv-str}) yields $\boldsymbol{\nu}_k\to\boldsymbol{\nu}_0$ strongly as $k\to\infty$, for $\{\boldsymbol{\nu}_k\}_{k\in\mathbb N}$ is strongly fundamental.
It has thus been established that $\{\boldsymbol{\nu}_k\}_{k\in\mathbb N}$ converges strongly to any of its $\mathbf A$-vague cluster points. As $\|\boldsymbol{\nu}_1-\boldsymbol{\nu}_2\|_{\breve{\mathcal E}^+_{\alpha}(\mathbf A)}$ is a metric, $\boldsymbol{\nu}_0$ has to be the unique $\mathbf A$-vague cluster point of $\{\boldsymbol{\nu}_k\}_{k\in\mathbb N}$. Since the
$\mathbf A$-vague topology is Hausdorff, $\boldsymbol{\nu}_0$ is
actually also the $\mathbf A$-vague limit of~$\{\boldsymbol{\nu}_k\}_{k\in\mathbb N}$
(cf.~\cite[Chapter~I, Section~9, n$^\circ$\,1]{B1}). This completes the proof.\hfill$\square$

\section{Proof of Theorem~\ref{th-main}}\label{sec-proof-main}

We start by observing that $\mathcal E^{\boldsymbol{\sigma}}_{\alpha,\mathbf f}(\mathbf A,\mathbf a,\mathbf g)$ is nonempty and, hence, (\ref{gauss-finite}) holds.
Indeed, it is seen from assumptions~(\ref{g-mass}) and~(b$'$) in consequence of~\cite[Lemma~1.2.2]{F1} that, for every $i=1,2$, there is a compact set $K_i\subset A_i$ such that $\langle g_i,\sigma^i|_{K_i}\rangle>a_i$ and $f_i(x)\leqslant M<\infty$ for all $x\in K_i$. Define $\theta^i:=\sigma^i|_{K_i}\bigl/{\langle g_i,\sigma^i|_{K_i}\rangle}$. Due to assumption~(c$'$) and Lemma~\ref{l-Ren}, we then obtain $\boldsymbol{\theta}:=(\theta^i)_{i=1,2}\in\mathcal E^{\boldsymbol{\sigma}}_{\alpha,\mathbf f}(\mathbf A,\mathbf a,\mathbf g)$ as claimed.

Therefore, the class $\mathbb M^{\boldsymbol{\sigma}}_{\alpha,\mathbf f}(\mathbf A,\mathbf a,\mathbf g)$ of all $\{\boldsymbol{\nu}_k\}_{k\in\mathbb N}\subset\mathcal E^{\boldsymbol{\sigma}}_{\alpha,\mathbf f}(\mathbf A,\mathbf a,\mathbf g)$ with
\begin{equation}\label{min-seq}\lim_{k\to\infty}\,G_{\alpha,\mathbf f}(\boldsymbol{\nu}_k)=G_{\alpha,\mathbf f}^{\boldsymbol{\sigma}}(\mathbf A,\mathbf a,\mathbf g)\end{equation}
is nonempty. Fix arbitrary $\{\boldsymbol{\nu}_k\}_{k\in\mathbb N}$ and $\{\boldsymbol{\mu}_m\}_{m\in\mathbb N}$ in $\mathbb M^{\boldsymbol{\sigma}}_{\alpha,\mathbf f}(\mathbf A,\mathbf a,\mathbf g)$. Taking~(\ref{inlusions}) into account, we proceed by proving that
\begin{equation}
\lim_{k,m\to\infty}\,\|\boldsymbol{\nu}_k-\boldsymbol{\mu}_m\|_{\breve{\mathcal E}^+_\alpha(\mathbf A)}=0. \label{fund}
\end{equation}
Based on the convexity of $\mathcal E^{\boldsymbol{\sigma}}_{\alpha,\mathbf f}(\mathbf A,\mathbf a,\mathbf g)$, from~(\ref{Re}) and~(\ref{wen}) we get
\[4G^{\boldsymbol{\sigma}}_{\alpha,\mathbf f}(\mathbf A,\mathbf a,\mathbf g)\leqslant4G_{\alpha,\mathbf f}\Bigl(\frac{\boldsymbol{\nu}_k+\boldsymbol{\mu}_m}{2}\Bigr)=
\|R\boldsymbol{\nu}_k+R\boldsymbol{\mu}_m\|^2_\alpha+
4\langle\mathbf f,\boldsymbol{\nu}_k+\boldsymbol{\mu}_m\rangle.\]
On
the other hand, applying the parallelogram identity in the
pre-Hilbert space $\mathcal E_\alpha(\mathbb R^n)$ to $R\boldsymbol{\nu}_k$
and $R\boldsymbol{\mu}_m$ and then adding and subtracting
$4\langle\mathbf{f},\boldsymbol{\nu}_k+\boldsymbol{\mu}_m\rangle$, we have
\[\|R\boldsymbol{\nu}_k-R\boldsymbol{\mu}_m\|^2_\alpha=
-\|R\boldsymbol{\nu}_k+R\boldsymbol{\mu}_m\|^2_\alpha-
4\langle\mathbf f,\boldsymbol{\nu}_k+\boldsymbol{\mu}_m\rangle
+2G_{\alpha,\mathbf f}(\boldsymbol{\nu}_k)+2G_{\alpha,\mathbf f}(\boldsymbol{\mu}_m).\]
When combined with the preceding relation, this gives
\begin{equation*}0\leqslant\|R\boldsymbol{\nu}_k-R\boldsymbol{\mu}_m\|^2_\alpha\leqslant-
4G^{\boldsymbol{\sigma}}_{\alpha,\mathbf f}(\mathbf A,\mathbf a,\mathbf g)+2G_{\alpha,\mathbf f}(\boldsymbol{\nu}_k)+2G_{\alpha,\mathbf f}(\boldsymbol{\mu}_m).\end{equation*}
On account of (\ref{isom}), (\ref{min-seq}) and the fact that $G_{\alpha,\mathbf f}^{\boldsymbol{\sigma}}(\mathbf A,\mathbf a,\mathbf g)$ is finite, we derive~(\ref{fund}) from the very relation by letting $k,m\to\infty$.

Assuming now $\{\boldsymbol{\nu}_k\}_{k\in\mathbb N}$ and $\{\boldsymbol{\mu}_m\}_{m\in\mathbb N}$ in~(\ref{fund}) to be equal, we see that any fixed sequence $\{\boldsymbol{\nu}_k\}_{k\in\mathbb N}\in\mathbb M^{\boldsymbol{\sigma}}_{\alpha,\mathbf f}(\mathbf A,\mathbf a,\mathbf g)$ is strongly fundamental in the metric space $\mathcal
E^+_\alpha(\mathbf A,\leqslant\!\mathbf a,\mathbf g)$. Thus, by Theorem~\ref{complete}, there exists the unique $\boldsymbol{\nu}_0\in\mathcal E^+_\alpha(\mathbf A,\leqslant\!\mathbf a,\mathbf g)$ such that
\begin{equation}\label{starr}\boldsymbol{\nu}_k\to\boldsymbol{\nu}_0\mbox{ \ $\mathbf A$-vaguely (as $k\to\infty$)},\end{equation}
\begin{equation}\label{starr1}\lim_{k\to\infty}\,\|\boldsymbol{\nu}_k-\boldsymbol{\nu}_0\|_{\breve{\mathcal E}^+_\alpha(\mathbf A)}=0.\end{equation}
We assert that this $\boldsymbol{\nu}_0$ gives a solution to Problem~\ref{pr2}, i.e.
\begin{equation}\label{solution}\boldsymbol{\nu}_0\in\mathcal E^{\boldsymbol{\sigma}}_{\alpha,\mathbf f}(\mathbf A,\mathbf a,\mathbf g)\text{ \ and \ }
G_{\alpha,\mathbf f}(\boldsymbol{\nu}_0)=G^{\boldsymbol{\sigma}}_{\alpha,\mathbf f}(\mathbf A,\mathbf a,\mathbf g).\end{equation}

Observe that
\begin{equation*}\label{str-vag}G_{\alpha,\mathbf f}(\boldsymbol{\nu}_0)\leqslant\liminf_{k\to\infty}\,G_{\alpha,\mathbf f}(\boldsymbol{\nu}_k).\end{equation*}
Indeed, if Case I holds, then this inequality can be obtained directly from~(\ref{starr}) and~(\ref{starr1}), while otherwise it follows from~(\ref{starr1}) with the help of~(\ref{C2}). Combining it with~(\ref{min-seq}) and~(\ref{gauss-finite}), we get
$G_{\alpha,\mathbf f}(\boldsymbol{\nu}_0)\leqslant G^{\boldsymbol{\sigma}}_{\alpha,\mathbf f}(\mathbf A,\mathbf a,\mathbf g)<\infty$.

As $\mathfrak M^{\boldsymbol\sigma}(\mathbf A)$ is $\mathbf A$-vaguely closed, we therefore conclude that
relation~(\ref{solution}) will have been established once for each $i=1,2$ we show
\begin{equation}\label{g}\langle g_i,\nu_0^i\rangle=a_i.\end{equation}
Consider an exhaustion of $A_i$ by an increasing sequence of compact sets $K_\ell\subset A_i$, $\ell\in\mathbb N$.
In view of the positivity and continuity of~$g_i$ on~$A_i$, from~(\ref{starr}) and \cite[Lemma~1.2.2]{F1} we get
\begin{align*}a_i&\geqslant\langle g_i,\nu_0^i\rangle=\lim_{\ell\to\infty}\,\bigl\langle g_i\chi_{K_\ell},\nu_0^i\bigr\rangle\geqslant\lim_{\ell\to\infty}\,\limsup_{k\to\infty}\,\bigl\langle g_i\chi_{K_\ell},\nu_{k}^i\bigr\rangle\\&{}
=a_i-\lim_{\ell\to\infty}\,\liminf_{k\to\infty}\,\bigl\langle g_i\chi_{A_i\setminus K_\ell},\nu_{k}^i\bigr\rangle.\end{align*}
Hence, to prove (\ref{g}), it is enough to verify the relation
\begin{equation}\label{g0}\lim_{\ell\to\infty}\,\liminf_{k\to\infty}\,\bigl\langle g_i\chi_{A_i\setminus K_\ell},\nu_{k}^i\bigr\rangle=0.\end{equation}
Since, by~(d$'$), \[\infty>\langle g_i,\sigma^i\rangle=\lim_{\ell\to\infty}\,\bigl\langle g_i\chi_{K_\ell},\sigma^i\bigr\rangle,\] we have
\[\lim_{\ell\to\infty}\,\bigl\langle g_i\chi_{A_i\setminus K_\ell},\sigma^i\bigr\rangle=0.\]
When combined with
\[\bigl\langle g_i\chi_{A_i\setminus K_\ell},\nu_{k}^i\bigr\rangle\leqslant\bigl\langle g_i\chi_{A_i\setminus K_\ell},\sigma^i\bigr\rangle\quad\text{for all \ }\ell,k\in\mathbb N,\]
this implies (\ref{g0}), hence (\ref{g}), and consequently~(\ref{solution}).

It is left to establish the statement on the uniqueness. Let, on the contrary, $\widehat{\boldsymbol{\nu}}_0$ be an other solution of Problem~\ref{pr2}. Then trivial sequences $\{\boldsymbol{\nu}_0\}$ and $\{\widehat{\boldsymbol{\nu}}_0\}$ are both elements of $\mathbb M^{\boldsymbol{\sigma}}_{\alpha,\mathbf f}(\mathbf A,\mathbf a,\mathbf g)$ and therefore, by~(\ref{fund}), $\|\boldsymbol{\nu}_0-\widehat{\boldsymbol{\nu}}_0\|_{\breve{\mathcal E}^+_\alpha(\mathbf A)}=0$. As $\breve{\mathcal E}^+_\alpha(\mathbf A)$ is a metric space, this results in $\boldsymbol{\nu}_0=\widehat{\boldsymbol{\nu}}_0$, and the proof is complete.\hfill$\square$

\vspace{5mm}


\end{document}